\DeclareSymbolFont{AMSb}{U}{msb}{m}{n}
\documentclass[11pt,a4paper,reqno, noamsfonts]{amsart}
\usepackage[utopia,  cal = cmcal]{mathdesign}

\def\ignore#1{\relax}

\ignore{
\usepackage{etoolbox}
\makeatletter
\def\new@mathgroup{\alloc@8\mathgroup\mathchardef\@cclvi}
\patchcmd{\document@select@group}{\sixt@@n}{\@cclvi}{}{}
\patchcmd{\select@group}{\sixt@@n}{\@cclvi}{}{}
\makeatother
}

\usepackage{bm}

\usepackage{youngtab}  

\usepackage{color}

\usepackage{amsthm}
\usepackage{enumitem}
\usepackage{fullpage}

\usepackage[svgnames]{xcolor}
\usepackage[plainpages=false, pdfpagelabels,  colorlinks=true, pdfstartview=FitV, linkcolor=DarkBlue, citecolor=DarkRed, urlcolor=blue]{hyperref}

\usepackage{tikz}

\ignore{
\renewcommand{\ge}{\geqslant}
\renewcommand{\ge}{\geqslant}

\renewcommand{\le}{\leqslant}

\renewcommand{\unrhd}{\trianglerighteqslant}

}

\theoremstyle{plain}
\newtheorem{theorem}{Theorem}[section]

\newtheorem{lemma}[theorem]{Lemma}
\newtheorem{corollary}[theorem]{Corollary}

\theoremstyle{definition}

\newtheorem{example}[theorem]{Example}

\numberwithin{equation}{section}

\DeclareMathOperator{\Res}{Res}
\DeclareMathOperator{\Ind}{Ind}

\DeclareMathOperator{\row}{row}

\DeclareMathOperator{\node}{node}

\DeclareMathOperator{\spn}{span}

\newcommand{\Std}{\mathrm{Std}}

\newcommand\std[1]{\mathcal{T}^\Std(#1)}

\newcommand{\cell}[3]{\Delta_{{#1}_{#2}}^{#3}}
\newcommand{\el}{l}

\newcommand\inv{^{-1}}

\input xy
\xyoption{all}

\def\mf{\mathfrak}
\def\la{\lambda}
\def\mft{\mathfrak t}
\def\mfs{\mathfrak s}

\def\mfv{\mathfrak v}

\def\power #1{^{(#1)}}

\def\ignore#1{\relax}

\begin{document}\usetikzlibrary{matrix,arrows,decorations.pathreplacing,backgrounds,decorations.markings}

\title[Cell filtration of restricted modules]
{A cell filtration of the restriction of a cell module}

\author{Frederick M. Goodman}
\address
{Department of Mathematics, University of Iowa, Iowa City, IA, USA}
\email{frederick-goodman@uiowa.edu}

\author[R. Kilgore]{Ross Kilgore}
\email{rosskilgore@gmail.com}

\author[N.Teff]{Nicholas Teff}
\email{nicholas\!.\!teff@gmail.com}

\begin{abstract}
We give a new proof that the restriction of a cell module of the Hecke algebra of the symmetric group  $\mathfrak S_n$ to the Hecke algebra of $\mathfrak S_{n-1}$  has a filtration by cell modules 
\end{abstract}

\thanks{We thank Andrew Mathas for detailed discussions of this problem.}

\maketitle

\section{Introduction}
Let $R$ be an integral domain and $q$ be a unit in $R$.   The  Iwahori--Hecke algebra of the symmetric group, denoted  $\mathcal{H}_n=\mathcal{H}_n(q^2)$, is the algebra presented by  generators $T_1,\ldots, T_{n-1}$, and  relations
\begin{align*}
&T_iT_j=T_jT_i,&&\text{if $|i-j|>1$,}\\
&T_iT_{i+1}T_{i}=T_{i+1}T_iT_{i+1},&&\text{for $i=1,\ldots,n-2$,}\\
&(T_i-q)(T_i+q^{-1})=0,&&\text{for $i=1,\ldots,n-1$.}
\end{align*}
If $v\in\mathfrak{S}_n$, and $v=s_{i_1}s_{i_2}\cdots s_{i_l}$ is a reduced expression for $v$ in $\mathfrak{S}_n$, then $T_v=T_{i_1}T_{i_2}\cdots T_{i_l}$ is well defined in $\mathcal{H}_n(q^2)$ and $\{T_v\mid v\in \mathfrak{S}_n\}$ is an $R$--basis of $\mathcal{H}_n(q^2)$.
It follows from this that $\mathcal H_n$ imbeds in $\mathcal H_{n+1}$ for all $n \ge 0$. 

The representation theory of $\mathcal{H}_n(q^2)$ was studied by Dipper and James in ~\cite{MR812444, MR872250},  generalizing the approach to the representation theory of the symmetric groups via Specht modules in ~\cite{MR513828}.   Murphy developed a new combinatorial approach to the representation theory of the Hecke algebras $\mathcal{H}_n(q^2)$  in \cite{MR1327362}.     Murphy's analysis involves certain elements $m^\la_{\mfs \mft}$ of the Hecke algebra, indexed by  Young diagrams $\la$ of size $n$ and a pairs of $\mfs, \mft$ of standard $\la$--tableaux.   Murphy shows that his elements satisfy the following properties:
\begin{enumerate}
\item The collection of elements $m^\la_{\mfs \mft}$, as $\la$ varies over Young diagrams of size $n$ and $\mfs, \mft$ vary over standard $\la$--tableaux,  is an $R$--basis of $\mathcal{H}_n(q^2)$.
\item   For $h \in \mathcal{H}_n(q^2)$,   $m^\la_{\mfs \mft} h =  \sum_\mfv  r_\mfv  m^\la_{\mfs \mfv}  + x$, where  the coefficients $r_v \in R$  depend only on $\mft$ and $h$  (and not on $\mfs$)   and $x$ is in the $R$--span of basis elements  $m^\mu_{\mathfrak{ x y}}$  where $\mu$ is greater than $\la$ in dominance order.  
\item  $(m^\la_{\mfs \mft})^*  =  m^\la_{\mft \mfs}$,  where $*$ is the algebra involution of $\mathcal{H}_n(q^2)$ defined by ${T_v}^*  = T_{v\inv}$.   
\end{enumerate}

Independent of Murphy and more or less simultaneously,  Graham and Lehrer~\cite{MR1376244}  introduced a theory of {\em cellular algebras}.  A cellular algebra is an algebra $A$ over an integral domain $R$ with an $R$--linear involution $*$,   and auxiliary data consisting of a finite partially ordered set $(\hat A,  \unrhd)$ and for each $\la \in \hat A$ a finite index set $\hat A^\la$.    $A$ is required to have an $R$--basis $a^\la_{\mfs \mft}$ indexed by elements $\la \in \hat A$ and  pairs $\mfs, \mft$ in $\hat A^\la$, having properties analogous to properties (2) and (3) of the Murphy basis, listed above.   The basis of $A$ is called a {\em cellular basis}.   Thus, in the language of Graham and Lehrer,  Murphy showed that his collection of elements $m^\la_{\mfs \mft}$ is a cellular basis of the Hecke algebra $\mathcal{H}_n(q^2)$. For the Hecke algebra $\mathcal H_n = \mathcal{H}_n(q^2)$,  the relevant partially ordered set $(\widehat{\mathcal{H}}_n, \unrhd)$ is the set of Young diagrams of size $n$ with dominance order and the index set $\widehat{\mathcal{H}}_n^\la$ is the set of standard tableaux of shape $\la$.  

Let $A$ be a cellular algebra with data $(*, \hat A, \unrhd)$ and cellular basis $\mathscr A = \{a^\la_{\mfs, \mft} :  \la \in \hat A \text{ and }  \mfs, \mft \in \hat A^\la\}$.    Fix $\la \in \hat A$ and define $\hat A^{\rhd \la}$ to be the span of basis elements $a^\mu_{\mathfrak{x y}}$  with $\mu \rhd \la$,  and  likewise
$\hat A^{\unrhd \la}$ to be the span of basis elements $a^\mu_{\mathfrak{x y}}$  with $\mu \unrhd \la$.  It follows that these are two--sided $*$--invariant ideals of $A$.    Moreover,  for any fixed $\mfs \in \hat A^\la$, 
$\Delta^\la  =  \spn\{a^\la_{\mfs \mft}  + A^{\rhd \la} :  \mft \in \hat A^\la\}$ is an $A$--submodule of $A^{\unrhd \la}/A^{\rhd \la}$,  which is free as an $R$--module.   Up to isomorphism, the {\em cell module} $\Delta^\la$ is independent of $\mfs$.     Graham and Lehrer show that the cell module has a canonical bilinear form.   When $R$ is specialized to be a field, the quotient of the cell module by the radical  of the bilinear form is either zero or simple, and all simple modules arise in this way ~\cite[Theorem 3.4]{MR1376244}.  

Let $A$ be  a cellular algebra and $M$ an $A$--module.   Say that $M$ has an {\em order preserving cell filtration}  if  $M$ has a filtration by $A$--submodules:
$$
(0) = M_0 \subset M_1 \subset \cdots M_{s-1}  \subset M_s = M,
$$
such that for each $i$, there exists $\la\power i \in \hat A$ such that $M_i/M_{i-1} \cong \Delta^{\la \power i}$, and, moreover, $\la\power 1 \rhd \la \power 2 \cdots  \rhd \la \power s$.  

This note concerns the following theorem regarding restrictions of cell modules of the Hecke algebras $\mathcal H_n = \mathcal H_n(q^2)$:
\begin{theorem}[Jost, Murphy]  \label{theorem cell filtration of restricted cell modules}
Let $n \ge 1$ and $\lambda$ be a Young diagram of size $n$.  Let $\cell {\mathcal H} n \lambda$ be the corresponding cell module of $\mathcal H_n$.  Then 
$\Res^{\mathcal H_n}_{\mathcal H_{n-1}}(\cell {\mathcal H} n \lambda)$    has an order preserving filtration by cell modules of $\mathcal H_{n-1}$.  
\end{theorem}

Jost ~\cite{MR1461487}  has shown, using the Dipper--James description of Specht modules of the Hecke algebras ~\cite{MR812444},
that the restriction of a Specht module has a filtration by Specht modules.  Together with Murphy's result that the cell modules of the Hecke algebras can be identified with the Specht modules~\cite[Theorem 5.3]{MR1327362}, this shows that the restriction of a cell module has a cell filtration.

It seems that  there is no complete proof of 
\hyperref[theorem cell filtration of restricted cell modules]{Theorem 
\ref*{theorem cell filtration of restricted cell modules}} in the literature based on Murphy's description of the cellular structure and representation theory of the Hecke algebras.   The proof given in ~\cite[Proposition 6.1]{MR1711316} has a gap, as it is not evident that the filtration of 
$\Res^{\mathcal H_n}_{\mathcal H_{n-1}}(\cell {\mathcal H} n \lambda)$  
provided there has subquotients isomorphic to cell modules of $\mathcal H_{n-1}$
The purpose of this note is to give a detailed proof of this fact.

Our treatment of this problem developed out of a correspondence with Andrew Mathas.  In the meanwhile, Mathas has also found a different approach, using the seminormal basis, which extends also to   the more general contexts of cyclotomic Hecke algebras and cyclotomic quiver Hecke algebras ~\cite{Mathas-private-communication-2015}.  

 \medskip
We remark that there is a companion theorem regarding cell filtrations of induced modules:
\begin{theorem}[Dipper--James, Murphy, Mathas]  \label{theorem:  cell filtration of induced cell modules for Hecke algebra}
Let $\mu$ be a Young diagram of size $n$ and let $\cell {\mathcal H} n \mu$ be the corresponding cell module of $\mathcal H_n$.  Then $\Ind_{\mathcal H_n}^{\mathcal H_{n+1}}(\cell {\mathcal H} n \mu)$  has an order preserving filtration by cell modules of $\mathcal H_{n+1}$.
\end{theorem}

Dipper and James showed that the induced module of a Specht module of $\mathcal H_n$ has a filtration by Specht modules of $\mathcal H_{n+1}$.  Together with Murphy's theorem  ~\cite[Theorem 5.3]{MR1327362}, this yields 
\hyperref[theorem:  cell filtration of induced cell modules for Hecke algebra]{Theorem \ref*{theorem:  cell filtration of induced cell modules for Hecke algebra}}.
A different proof was recently given by Mathas ~\cite{MR2531227};  this proof is based on Murphy's theorem ~\cite[Theorem 7.2]{MR1327362}   on the existence of a cell filtration of permutation modules of $\mathcal H_n$.

\section{Preliminaries}
We will assume familiarity with the usual combinatorial notions related to the representation theory of the symmetric groups and their Hecke algebras.  We refer to ~\cite[Ch.~3]{MR1711316}  for details and notation.  

For purposes of this note, for a Young diagram $\la$ of size $n$,  a tableau will mean an assignment of the 
numbers $1, \dots, n$ to the nodes of $\la$.  We let $\mathcal T(\la)$ denote the set of all $\la$--tableaux and $\mathcal{T}^{\Std}(\lambda)$ the set of standard $\la$--tableaux.
We will let $\mft^\la$ denote the ``superstandard"  tableau of shape $\la$, in which the numbers $1$  through $n$  are
entered in increasing order from left to right along the rows of
$[\lambda]$. Thus when $n=6$ and
$\lambda=(3,2,1)$,
\begin{align}\label{tabex1}
\mathfrak{t}^\lambda=\text{\tiny\Yvcentermath1$\young(123,45,6)$}\,.
\end{align}
 We regard the symmetric group  $\mathfrak S_n$  as acting on the right on the set of numbers $1, \dots n$.
 The symmetric group $\mathfrak S_n$  acts  on  the set of tableaux of size $n$, by acting on the entries;  this action is free and transitive.    The \emph{Young subgroup} $\mathfrak{S}_\lambda$ is defined to be the row stabiliser of $\mathfrak{t}^\lambda$ in $\mathfrak{S}_{n}$.
For each $\mft \in \mathcal T(\lambda)$,  let  $w(\mft)$ denote the unique permutation such that
$\mft = \mft^\lambda w(\mft)$. 

If $\la$ is a Young diagram of size $n$,  let  $m_\la=\sum_{v\in\mathfrak{S}_\la}q^{\el(v)}T_v$,  where $\el$ denotes the length function on $\mathfrak S_n$.     For standard tableaux $\mfs, \mft$ of shape $\la$,  let 
$$
m^\la_{\mfs \mft} =  T_{w(\mathfrak{s})}^*m_\lambda T_{w(\mathfrak{t})}.
$$
These are Murphy's basis elements, as described in the introduction.    Define
$$
m^\la_{\mft}    =      m_\la  T_{w(\mft)}  +  \mathcal{H}^{\rhd \la}.
$$
The set  $\{m^\la_\mft :  \mft  \text{ is a standard $\la$--tableau}\}$ is an $R$--basis of the cell module
$\Delta^\la$.  We will write $M^\la$ for the ``permutation module"
$$
M^\la =  m_\la   \mathcal{H}_n.
$$

Let $\lambda$ be a Young diagram of size $n$  and let $\mft \in \mathcal T(\lambda)$.  Call a node of $\lambda$ {\em addable} if the addition of the node to $\la$ yields a Young diagram of size $n+1$.  Define a {\em removable node} similarly.  Let $\alpha$ be an addable node of $\lambda$.  Then we write $\mft \cup \alpha$ for the tableau of shape $\lambda \cup \alpha$ which agrees with $\mft$ on the nodes of  $\lambda$ and which has the entry $n+1$ in node $\alpha$. 
If $\mft$ is a standard $\lambda$--tableau and $1 \le k < n$,  let 
$\mft \downarrow_k$ denote the tableau obtained by deleting from $\mft$ the nodes containing $k+1, \dots, n$;  then  $\mft \downarrow_k$ is a standard tableau of size $k$.

For  $1 \le i,j \le n$,  let
\begin{align*}
T_{i,j}=
\begin{cases}
T_iT_{i+1}\cdots T_{j-1} = T_{(j, j-1, \dots,  i)},&\text{if $j\ge i$,}\\
T_{i-1}T_{i-2}\cdots T_{j} =  T_{(j, j+1, \dots, i)},&\text{if $i>j$.}
\end{cases}
\end{align*}

\begin{lemma}  \label{permutation of  tableau adjoined addable node}
Let $\lambda$ be a Young diagram of size $n$,  let $\alpha$ be a removable node of $\lambda$, and let $\mu = \lambda \setminus \alpha$.   Let $a$ be the entry of  $t^\lambda$ in the node $\alpha$.   
Let $\mfs \in \mathcal T(\mu)$ be a $\mu$--tableau.
Then
$$
w(\mfs \cup \alpha) = (n, n-1, \dots, a)  \,w(\mfs), 
$$
and
$$
T_{w(\mfs \cup \alpha)} = T_{(n, n-1, \dots, a)} T_{w(\mfs)} = T_{a, n} T_{w(\mfs)}.
$$
\end{lemma}

\begin{proof}  We have
$$
\mfs \cup \alpha =  (\mft^\mu \cup \alpha) w(\mfs) = \mft^\lambda  (n, n-1, \dots, a) \, w(\mfs).
$$
Therefore, $$w(\mfs \cup \alpha) = (n, n-1, \dots, a)  w(\mfs).$$
Now one can check that $(n, n-1, \dots, a)$ is a distinguished left coset representative of 
$\mathfrak S_{n-1}$ in $\mathfrak S_n$.  Therefore, 
$$
T_{w(\mfs \cup \alpha)} =  T_{(n, n-1, \dots, a)}  T_{w(\mfs)} = T_{a, n} T_{w(\mfs)}.
$$
\end{proof}

\begin{lemma}  \label{lemma:  curious identity}
Let $\lambda$ be a Young diagram of size $n$,  let $\alpha$ be a removable node of $\lambda$, and let $\mu = \lambda \setminus \alpha$.     Let $r$ be the row index of $\alpha$ and let $b$ and $a$ be the first and last entry in the $r$--th row of the standard tableaux $\mft^\lambda$.  
Write
$$
D(\alpha) = 1 + q T_{a-1} + q^2 T_{a-1} T_{a-2} + \cdots + q^{a-b} T_{a-1} T_{a-2} \cdots T_b.
$$
Then
\begin{equation}  \label{relate m mu with m lambda when mu arrow lambda}
D(\alpha)^* T_{a, n} m_\mu =  m_\lambda T_{a, n}.  
\end{equation}
\end{lemma}

\begin{proof}
Let $\la'$ be the composition $\la' = (\mu_1, \dots, \mu_r, 1, \mu_{r+1}, \dots, \mu_\el)$.   One has $T_{n, a}\inv  T_j  T_{n, a}  = T_{j+1}$  if
 $a \le j \le n-1$.  This follows from the identity in the braid group:
 $$
( \sigma_{a}\inv \cdots \sigma_{n-1}\inv) \sigma_j (\sigma_{n-1} \cdots \sigma_{a}) = \sigma_{j+1},
 $$
 for $a \le j \le n-1$,  where the elements $\sigma_i$ are the Artin generators of the braid group.
 From this, we obtain:
 $$
 m_{\la'} =  T_{n, a}\inv  m_\mu  T_{n, a}.
 $$
 Note that $\mathfrak S_{\la'} \subset \mathfrak S_\la$  and $D(\alpha) = \sum  q^{\el(x)} T_x$, as where the sum is over the distinguished right coset representatives of 
 $\mathfrak S_{\la'}$ in  $\mathfrak S_\la$.  Hence $m_\la = m_{\la'} D(\alpha)$, and the result follows.
\end{proof}

\section{Proof of {Theorem \ref{theorem cell filtration of restricted cell modules}}}
 \label{appendix: cell filtrations}
In this section we  give a proof of \hyperref[theorem cell filtration of restricted cell modules]{Theorem 
\ref*{theorem cell filtration of restricted cell modules}}.   The proof is based on Murphy's fundamental paper ~\cite{MR1327362}, but we refer specifically to  Mathas' reworking of Murphy's theory in  ~\cite[Ch.~3]{MR1711316}.   

\def\mfg{\mathfrak g}
Let $\la$ be a Young diagram of size $n$.  
We recall the definition of a Garnir tableau of shape $\lambda$, see ~\cite[page~33]{MR1711316}.  Suppose  both $(i, j)$ and $(i+1, j)$ are nodes of $\lambda$.   The $(i, j)$--Garnir strip
consists of all nodes of  $\lambda$  in row $i$ weakly to the right of $(i, j)$ together with all nodes
in row $(i+1)$ weakly to the left of $(i+1, j)$.   Let $a$ be the entry of $\mft^\lambda$ in the node $(i, j)$ and $b$ the entry of $\mft^\lambda$ in the node $(i+1, j)$.  
The $(i, j)$--Garnir tableau $\mfg = \mfg_{(i,j)}$ is the (row standard) tableau which agrees with $\mft^\lambda$ outside the Garnir strip, and in which the numbers $a, a+1, \dots, b$ are entered from left to right in the Garnir strip, first in row $i+1$ and then in row $i$. 

\begin{example} Let $\la = (3, 2, 1)$ then 
\begin{align}
\mfg_{(1,1)}=\text{\tiny\Yvcentermath1$\young(234,15,6)$}, \quad \mfg_{(1,2)}=\text{\tiny\Yvcentermath1$\young(145,23,6)$}, \, \text{and} \quad \mfg_{(2,1)}=\text{\tiny\Yvcentermath1$\young(123,56,4)$}\,.
\end{align}
\end{example}

Fix $\lambda$ a Young diagram of size $n$ and a Garnir tableau $\mfg = \mfg_{(i, j)}$ of shape $\lambda$.  
All of the row standard $\lambda$--tableaux which agree with $\mft^\lambda$ outside the $(i, j)$--Garnir strip, apart from $\mfg$,  are in fact standard.  Moreover, a standard $\lambda$--tableau $\tau$ agrees with $\mft^\lambda$ outside the Garnir strip if and only if $\tau \rhd \mfg$.    Define
\begin{equation}
h_\mfg = m_\lambda T_{w(\mfg)} + \sum_{\tau \rhd \mfg}  m_\lambda T_{w(\tau)},
\end{equation}
where the sum is over standard $\lambda$--tableaux $\tau \rhd \mfg$.  
It follows from  ~\cite[Lemma 3.13]{MR1711316} and the cellularity of the Murphy basis that $h_\mfg$ is an element of $M^\lambda \cap \mathcal H_n^{\rhd \lambda}$.   Let $M_0^\lambda$ be the  right
$\mathcal H_n$--module generated by the elements $h_\mfg$, as $\mfg$ varies over all Garnir tableaux of shape $\lambda$.    Then we have $M_0^\lambda \subseteq M^\lambda \cap \mathcal H_n^{\rhd \lambda}$.

\begin{lemma}  \label{Garnir characterization of M lambda intersect H greater than lambda}
$M_0^\lambda = M^\lambda \cap \mathcal H_n^{\rhd \lambda}$.
\end{lemma}

\begin{proof} By the proof of~\cite[Lemma 3.15]{MR1711316}, if $\mft$ is a row standard $\lambda$--tableau that is not standard, then $m_\lambda T_{w(\mft)}  = x + h$,  where
$x$ is a linear combination of Murphy basis elements $m_\lambda T_{w(\mfv)}$  (with $\mfv \in \std \lambda$) and 
$h \in M_0^\lambda$.    Thus we have
$$
\begin{aligned}
M^\lambda &= \spn\{m_\lambda T_{w(\mfv)}: v \in \std \lambda\} + M_0^\lambda.
\end{aligned}
$$  
It follows from this that $ M^\lambda \cap \mathcal H_n^{\rhd \lambda} \subseteq M_0^\lambda$. 
\end{proof}

\begin{lemma}  \label{lemma: a submodule lemma}
Let $\lambda$ be a Young diagram of size $n$.  
 Let $\mathcal S$ be a subset of the set of row standard $\lambda$--tableaux  and let $I$ be a subset of $\{1, 2, \dots, n-1\}$, with the following properties:
 \begin{enumerate}
 \item  If $\mfs \in \mathcal S$ and $\mft$ is a row standard $\lambda$--tableau with $\mft \unrhd \mfs$,  then $\mft \in \mathcal S$.
 \item  If $\mfs \in \mathcal S$ and $i \in I$, with $i$ and $i+1$ in different rows of $\mfs$,  then $\mfs s_i \in \mathcal S$.
 \end{enumerate}
 Let $\mathcal H_I$  be the unital subalgebra of $\mathcal H_n$ generated by $\{T_i : i \in I\}$.  Then
 $$
 M = \spn\{ m_\lambda T_{w(\mft)} : \mft \in \mathcal S \cap \std \lambda\} +  (M^\lambda \cap \mathcal H_n^{\rhd \lambda})
 $$
 is a right $\mathcal H_I$--submodule of $M^\lambda$.
\end{lemma}

\begin{proof}  Let $\mft \in \mathcal S \cap \std \lambda$ and let $i \in I$.   We have to show that
$m_\lambda T_{w(t)} T_i \in M$.    If $i$ and $i+1$ are in the same row of $\mft$, then
$m_\lambda T_{w(t)} T_i  = q^2 m_\lambda T_{w(t)}$.  If $i$ and $i+1$ are in different rows and different columns, then $\mft' = \mft s_i \in \mathcal S \cap \std \lambda$ by hypothesis, and
$m_\lambda T_{w(t)} T_i  $ is a linear combination of $m_\lambda T_{w(t)}$ and $m_\lambda T_{w(t')}$.
Finally, if $i$ and $i+1$ are in the same column of $\mft$,  then $\mft' = \mft s_i$ is a row standard but not standard,  $\mft' \in \mathcal S$ by hypothesis, and $m_\lambda T_{w(t)} T_i  = m_\lambda T_{w(t')} $.   By ~\cite[Lemma 3.15]{MR1711316} and cellularity of the Murphy basis, 
$m_\lambda T_{w(t')} = x+ h$  where $x$ is a linear combination of elements $m_\lambda T_{w(\mfv)}$ with $\mfv$ standard and $\mfv \rhd \mft'$,  and $h \in M^\lambda \cap \mathcal H_n^{\rhd \lambda}$.  
By hypothesis, each such $\mfv$ is in $\mathcal S$,  so
$m_\lambda T_{w(t)} T_i  = m_\lambda T_{w(t')} \in M$.  
\end{proof}

Let $\la$ be a Young diagram of size $n$ and let $\mft \in \std \la$.  For any $1 \le i \le n$, let $\row_\mft(i)$ denote the row in which $i$ appears in $\mft$.  

\begin{corollary}  \label{H n-1 invariance of tableaux with row index of n large}
Let $\lambda$ be a Young diagram of size $n$ and let $r \ge 1$.  Then
$$
\spn\{ m_\lambda T_{w(t)} : \mft \in \std \lambda \text{ and } 
\row_\mft(n) \ge r\} +  (M^\lambda \cap \mathcal H_n^{\rhd \lambda})
$$
is a right $\mathcal H_{n-1}$--submodule of $M^\lambda$.
\end{corollary}

\begin{proof} In \hyperref[lemma: a submodule lemma]{Lemma \ref*{lemma: a submodule lemma}},  take $\mathcal S$ to be the set of row standard tableaux $\mfs$ such that $\row_\mfs(n) \ge r$ and take $I = \{1, 2, \dots, n-2\}$.  
\end{proof}

In the following discussion,  $[\mfs]$ denotes the shape of a standard tableau $\mfs$.

\begin{corollary}  \label{corollary: another submodule theorem}
Let $\lambda$ be a Young diagram of size $n$ and let $\gamma$ be a node of $\lambda$.  Let $m$ denote the entry of $\mft^\la$ in node $\gamma$.
   Let $\mathcal H_{m, n}$ be the unital subalgebra of $\mathcal H_n$ generated by $\{T_{m}, \dots, T_{n-1}\}$.  Then
$$
\spn\{ m_\lambda T_{w(t)} :  \mft \in \std \lambda \text{ and }   [\mft \downarrow_{\,m-1}] =  [\mft^\lambda \downarrow_{\,m-1}] \}
+  (M^\lambda \cap \mathcal H_n^{\rhd \lambda})
$$
is a right $\mathcal H_{m, n}$--submodule of $M^\lambda$.
\end{corollary}

\begin{proof}  In \hyperref[lemma: a submodule lemma]{Lemma \ref*{lemma: a submodule lemma}},  take $\mathcal S$ to be the set of row standard tableaux $\mfs$ such that  $[\mfs \downarrow_{\,m-1}] =  [\mft^\lambda \downarrow_{\,m-1}]$ and take $I = \{m, \dots, n-1\}$. 
\end{proof}

For the remainder of this section, we fix $n \ge 1$ and a Young diagram $\lambda$ of size $n$.
Let $\alpha_1, \dots, \alpha_p$ be the list of removable nodes of $\lambda$, listed from bottom to top,
and let $\mu\power j = \lambda \setminus \alpha_j$.    Let $N_0 = (0)$ and for $1 \le j \le p$, let 
$N_j$ be the $R$--submodule of $\cell {\mathcal H} n \lambda$ spanned by by the basis elements
$m^\lambda_\mft$ such that $\node_\mft(n) \in \{\alpha_1, \dots, \alpha_j\}$.
Then we have
$$
(0) = N_0 \subseteq N_1 \cdots \subseteq N_p = \Res^{\mathcal H_n}_{\mathcal H_{n-1}}(\cell {\mathcal H} n \lambda).
$$
The explicit form of the assertion of \,\hyperref[theorem cell filtration of restricted cell modules]{Theorem \ref*{theorem cell filtration of restricted cell modules}} is that the $N_j$  are $\mathcal H_{n-1}$--submodules of $ \Res^{\mathcal H_n}_{\mathcal H_{n-1}}(\cell {\mathcal H} n \lambda)$ and $N_j/N_{j-1} \cong \cell {\mathcal H} {n-1} {\mu \power j}$  for
$1 \le j \le p$.   The isomorphism is determined by 
\begin{equation} m^{\mu \power j}_\mfs \mapsto   m^\lambda_{\mfs \cup \alpha_j} + N_{j-1}.
\end{equation}

\begin{corollary}  For each $j$,  $N_j$ is a right $\mathcal H_{n-1}$-submodule of 
$\cell {\mathcal H} n \lambda$.
\end{corollary}

\begin{proof}  Immediate from \hyperref[H n-1 invariance of tableaux with row index of n large]{Corollary \ref*{H n-1 invariance of tableaux with row index of n large}}.
\end{proof}

Our goal is to show that $N_j/N_{j-1} \cong \cell {\mathcal H} {n-1} {\mu\power j}$  as $\mathcal H_{n-1}$--modules,
for each $j \ge 1$.

We fix one removable node  $\alpha = \alpha_k$  of $\lambda$, and write $\mu = \lambda \setminus \alpha$.   Let $D(\alpha)$ be defined as in the statement of \hyperref[lemma:  curious identity]{Lemma \ref*{lemma:  curious identity}}.
For a $\mathcal H_n$ module $M$,  we will  write $\Res(M)$ for $\Res^{\mathcal H_n}_{\mathcal H_{n-1}}(M)$.
Because of  Equation \eqref{relate m mu with m lambda when mu arrow lambda}, we
have an $\mathcal H_{n-1}$--module homomorphism from $M^\mu$ to $\Res(M^\lambda)$ defined by 
$$\varphi_0 : m_\mu h \mapsto  m_\lambda T_{a, n} h.$$    
If $\mfs \in \mathcal T(\mu)$ is a $\mu$--tableau, then we have
$$
\varphi_0(m_\mu T_{w(\mfs)}) =  m_\lambda T_{a, n} T_{w(\mfs)} = m_\lambda T_{w(\mfs \cup \alpha)},
$$
by \hyperref[permutation of  tableau adjoined addable node]{Lemma \ref*{permutation of  tableau adjoined addable node}}.   

  Let $\varphi$ be the composite homomorphism
$$
\varphi:  M^\mu \stackrel{\varphi_0}{\longrightarrow}  \Res(M^\lambda) \to 
\Res(\cell {\mathcal H} n \lambda)  \to  \Res(\cell {\mathcal H} n  \lambda)/N_{k-1}, 
$$
where the latter two maps are canonical quotient maps.  We claim that $\varphi$ factors through
$\cell {\mathcal H} {n-1} \mu$.    Because of 
\hyperref[Garnir characterization of M lambda intersect H greater than lambda]{Lemma \ref*{Garnir characterization of M lambda intersect H greater than lambda}}, 
it suffices to show that if $\mfg_0$ is a Garnir tableau of shape $\mu$, then $\varphi(h_{\mfg_0}) = 0$. 

Let $\mfg_0$  be the $(i, j)$--Garnir tableau of shape $\mu$  and let $\mfg$ be the $(i, j)$--Garnir tableau of shape $\lambda$.  There are two cases to consider:

{\em Case 1.}  The node $\alpha$ is not in the Garnir strip of $\mfg$.    In this case, there is a one to one correspondence between row standard tableaux $\tau_0$ of shape $\mu$ such that $\tau_0 \unrhd \mfg_0$, and row standard tableaux $\tau$ of shape $\lambda$ such that $\tau \unrhd \mfg$, given by
\begin{equation} \label{tau and tau zero 1}
\tau (n, n-1, \dots, m) = \tau_0 \cup \alpha,
\end{equation} 
where $m = \mf g(\alpha) = \mft^\lambda(\alpha)$. 
We claim that (when $\tau$ and $\tau_0$ are so related)
\begin{equation}     \label{tau and tau zero 2}
T_{w(\tau_0 \cup \alpha)}  =  T_{w(\tau)} T_{m, n}.
\end{equation}
In fact,  one can check that
$$
\tau \rhd \tau s_m \rhd \tau s_m s_{m+1} \rhd \cdots \rhd \tau (s_m s_{m+1} \cdots s_{n-1}) = \tau_0 \cup \alpha,
$$
and \eqref{tau and tau zero 2} follows. 
Now we have
\begin{equation}
\varphi_0(h_{\mfg_0}) =  m_\lambda \left(\sum_{\tau_0 \unrhd \mfg_0} T_{w(\tau_0 \cup \alpha)}\right)
= m_\lambda \left(\sum_{\tau \unrhd \mfg} T_{w(\tau)} \right ) T_{m, n} 
= h_\mfg T_{m, n}.
\end{equation}
Thus  $\varphi_0(h_{\mfg_0}) \in M^\lambda \cap \mathcal H_n^{\rhd \lambda}$, so $\varphi(h_{\mfg_0}) = 0$.   

{\em Case 2.}    The node $\alpha$ is in the Garnir strip of $\mfg$.  Let $m = \mfg(\alpha)$,  the largest entry in the Garnir strip of $\mfg$.   The row standard tableaux $\tau$ such that $\tau \unrhd \mfg$ either have $\node_\tau(m) = \alpha$  or $\node_\tau(m) =  (i+1, j)$.    Let $A$ be the set of $\tau$  such that
$\node_\tau(m) = \alpha$  and let $B$ be the set of $\tau$ such that $\node_\tau(m) =  (i+1, j)$. 
The set $A$ is in one to one correspondence with the set of row standard tableaux $\tau_0$ of shape $\mu$ with $\tau_0 \unrhd \mfg_0$; the correspondence is given by 
\begin{equation*} \label{tau and tau zero 3}
\tau (n, n-1, \dots, m) = \tau_0 \cup \alpha.
\end{equation*} 
For $\tau$ and $\tau_0$ so related we have
\begin{equation*}     \label{tau and tau zero 4}
T_{w(\tau_0 \cup \alpha)}  =  T_{w(\tau)} T_{m, n}.
\end{equation*}
Thus we have
\begin{equation}
\begin{aligned}
\varphi_0(h_{\mfg_0}) &=  m_\lambda \left(\sum_{\tau_0 \unrhd \mfg_0} T_{w(\tau_0 \cup \alpha)}\right)
= m_\lambda \left(\sum_{\tau  \in A} T_{w(\tau)} \right ) T_{m, n} \\
&= h_\mfg T_{m, n} -  m_\lambda \left(\sum_{\tau  \in B} T_{w(\tau)} \right ) T_{m, n} 
\end{aligned}
\end{equation}
If $\tau \in B$,   then $[\tau \downarrow_{m-1}] = [\mft^\lambda \downarrow_{m-1}]$.    By
\hyperref[corollary: another submodule theorem]{Corollary \ref*{corollary: another submodule theorem}},   for $\tau \in B$, 
$m_\lambda T_{w(\tau)}  T_{m, n}  =  x + h$,  where  $h \in M^\lambda \cap \mathcal H_n^{\rhd \lambda}$ and 
$x$ is a linear combination of Murphy basis elements
$m_\lambda T_{w(\mfv)}$  with $\mfv \in \std \lambda$ and 
$[\mfv \downarrow_{m-1}] = [\mft^\lambda \downarrow_{m-1}]$.    For such $\mfv$,  
$\row_\mfv(n)  \ge i+1$,  so the node of $n$ in $\mfv$ is one of $\alpha_1, \dots, \alpha_{k-1}$.  
Hence, $\varphi_0(h_{\mfg_0})$ is contained in
$$
M^\lambda \cap \mathcal H_n^{\rhd \lambda}  + \spn\{m_\lambda T_{w(\mfv)} : \mfv \in \std \lambda \text{ and }  \node_\mfv(n) \in \{\alpha_1, \dots, \alpha_{k-1}\} \}.
$$
It follows that 
 $\varphi(h_{\mfg_0}) = 0$.  This completes the proof that
$\varphi$ factors through $\cell {\mathcal H} {n-1} \mu$.  

The map $\bar \varphi: \cell {\mathcal H} {n-1} \mu \to \Res(\cell  {\mathcal H} {n} \lambda)/N_{k-1}$  determined by $\varphi$ satisfies
$$
\bar\varphi(m^\mu_\mft) =  m^\lambda_{\mft \cup \alpha} + N_{k-1},
$$
so has range $N_k/N_{k-1}$, and is an isomorphism onto its range.  This completes the proof of \hyperref[theorem cell filtration of restricted cell modules]{Theorem 
\ref*{theorem cell filtration of restricted cell modules}}.



\def\MR#1{\href{http://www.ams.org/mathscinet-getitem?mr=#1}{MR#1}}
\def\Dbar{\leavevmode\lower.6ex\hbox to 0pt{\hskip-.23ex \accent"16\hss}D}
\providecommand{\bysame}{\leavevmode\hbox to3em{\hrulefill}\thinspace}
\providecommand{\MR}{\relax\ifhmode\unskip\space\fi MR }
\providecommand{\MRhref}[2]{%
  \href{http://www.ams.org/mathscinet-getitem?mr=#1}{#2}
}
\providecommand{\href}[2]{#2}


\end{document}